\documentclass[11pt,a4paper]{article}
\usepackage{indentfirst}
\usepackage{amsmath}
\usepackage{amsfonts}
\usepackage{amsthm}
\usepackage{amssymb}
\usepackage{graphics}
\usepackage{graphicx}
\usepackage{multicol}

\def\N{\mathbb{N}}
\def\f{\bold{f}}
\def\g{\bold{g}}
\def\SS{\mathbb{S}}
\def\P{\mathbb{P}}
\def\E{\mathbb{E}}

\def\i{\mathbf{i}}

\def\RRR{\mathbb{R}}

\def\t{\textrm}
\def\w{\widetilde}

\def\b{\mathbf}

\def\ind{{\mathchoice {\rm 1\mskip-4mu l} {\rm 1\mskip-4mu l}
{\rm 1\mskip-4.5mu l} {\rm 1\mskip-5mu l}}}
\def\d{\text{d}}

\newcommand {\GG}{\mathbb{G}}
\newcommand {\TT}{\mathbb{T}}

\newcommand{\be} {\begin{equation}}
\newcommand{\ee} {\end{equation}}
\newcommand{\bea} {\begin{eqnarray}}
\newcommand{\eea} {\end{eqnarray}}
\newcommand{\Bea} {\begin{eqnarray*}}
\newcommand{\Eea} {\end{eqnarray*}}

\newtheorem{Thm}{Theorem}
\newtheorem*{Thm*}{Theorem}
\newtheorem{Lem}{Lemma}

\newtheorem{Pte}{Proposition}

\theoremstyle{definition} 
\theoremstyle{definition} \newtheorem*{key}{Key words}
\theoremstyle{definition} \newtheorem*{ams}{A.M.S. Classification}
\theoremstyle{remark}\newtheorem*{ex}{Example}
\theoremstyle{remark}
\theoremstyle{remark}\newtheorem{fig}{Figure}

\begin{document}
\title{Cell contamination and branching process in random environment with immigration}
\author{Vincent Bansaye \footnote{Laboratoire de Probabilités et Modèles Aléatoires. Université Pierre et Marie Curie et C.N.R.S. UMR 7599. 175, rue du
Chevaleret, 75 013 Paris, France.
$\newline$
\emph{e-mail} :  vincent.bansaye@upmc.fr }}
\maketitle \vspace{3cm}
\begin{abstract}
We consider  a branching model for a population of dividing cells infected by  parasites. Each cell receives parasites
by inheritance from its mother cell and independent contamination from outside the cell population. Parasites multiply randomly inside the cell and are shared  randomly between the two daughter cells when the cell divides. The law of the number of parasites which contaminate a given cell depends only on whether  the cell is already infected or not.
We determine first the asymptotic behavior of    branching processes in random environment with state dependent immigration, which gives the convergence in distribution of the number of parasites in a cell line. We then derive  a law of large numbers for the asymptotic proportions of cells with a given number of parasites. The main tools are  branching processes in random environment  and  laws of large numbers for Markov tree.
\end{abstract}
\begin{key}  Branching processes in random environment with immigration (IBPRE). Markov chain indexed by a tree.
 Empirical measures. Renewal theorem.
\end{key}
\begin{ams}60J80, 60J85, 60K37, 92C37, 92D25, 92D30.
\end{ams}

\section{Introduction}
We consider the following model for cell division with parasite infection and state dependent
contamination. The cell population starts from one single cell and divides in discrete time. At each generation, 
\\
(i) the parasites multiply randomly inside the cells,\\
(ii) each cell is contaminated by a random number of parasites which come from outside the cell population, \\
(iii)  each  cell divides into two daughter cells and
the  parasites are  shared randomly   into the two daughter cells.  \\

It is convenient to distinguish a first daughter cell called $0$ and a second one called $1$.
 We denote by  $\TT=\cup_{n \in \N} \{0,1\}^{n}$ the binary genealogical tree of the cell population, by $\GG_n$
 the set of cells in generation
$n$  and  by $Z_{\mathbf{i}}$ the number of parasites of  cell
 $\b{i}\in \TT$. We  write then $\b{i}0$ and $\b{i}1$ the two daughter cells of the cell $\b{i}\in\TT$. \\
 
 First, we describe by a branching process the random multiplication  and sharing of parasites in the cell, i.e. this branching process combines (i) and (iii). Second, we describe the random contamination (ii) by  immigration. Finally, we combine both in an i.i.d. manner  to fully describe the model. \\

\underline{I Parasite infection and cell division}
For every cell, we choose randomly  a mechanism for   multiplication
of the  parasites inside and  sharing of their
offspring when the cell divides. This mechanism is independent  and identically distributed for every cell. Its distribution is specified
by a random couple probability generating function (p.g.f)
 $\f$. This means that $\f$ is a.s. the p.g.f of a pair of random variables taking values in $\N$. \\
 
 More precisely let $(f_{\i})_{\i \in \TT}$  be a sequence of i.i.d. couple p.g.f  distributed
 as $\f$. For each  cell $\i$, $f_{\i}$  gives  
  the  reproduction law and sharing of the  offspring of  its parasites in the following way. For every $\i\in\TT$,   let
$(X^{(0)}_k(\i), X^{(1)}_k(\i))_{k\in\N}$ be a sequence of r.v. such that conditionally
on $\bold{f}_{\i}=\bold{g}$, $(X^{(0)}_k(\i), X^{(1)}_k(\i) )_{k\in\N}$ are i.i.d.  with common
couple p.g.f $\bold{g}$:
$$\forall \i\in\TT, \ \forall k \in \N, \ \forall s,t \in [0,1], \quad \E(s^{X^{(0)}_k(\i)}t^{X^{(1)}_k(\i)} \ \vert \ \f_{\i}=\bold{g})=\bold{g}(s,t).$$
Then, in each generation, each parasite $k$
of the cell $\i$ gives birth
 to $X^{(0)}_k (\i)+X^{(1)}_k(\i)$ children, $X^{(0)}_k(\i)$ of which  go
into the first daughter cell  and $X^{(1)}_k(\i)$ of which go into the second one, when the cell divides.
This is a more general model for parasite infection and cell division than the model
studied in \cite{vbk}, where there was no random environment ($\f$ was deterministic) and the the total number of parasites was a Galton Watson process. See $\cite{kim}$ for the original model in continuous time. \\

Our model   includes also the two following natural models, with random binomial repartition
of parasites. Let $Z$ be a random variable in $\N$ and
$(P_{\i})_{\i \in \TT}$ be i.i.d. random variable in $[0,1]$. In each generation, every parasite
 multiplies independently with the same reproduction law $Z$. Thus parasites follow a Galton Watson process. Moreover
 $P_{\i}$ gives the
mean fraction of parasites of the cell $\i$ which goes into the first daughter cell
when the cell divides. More precisely, conditionally on $P_{\i}=p$,
every parasite of the mother cell $\i$  chooses independently the first  daughter cell with  probability
 $p$ (and  the second one with probability $1-p$). \\
 It contains also the following model. Every parasite
gives  birth independently to a random cluster of parasites of size $Z$ and conditionally on $P_{\i}=p$,
every cluster of parasite  goes independently into the  first cell with
probability $p$  (and into the second one with probability $1-p$). \\

We want to take into account asymmetric repartition  of parasites and do not make any assumption about $\f$. Indeed  unequal sharing have been observed  when the cell divides, see e.g. experiments of M. de
Paepe,  G. Paul and F.  Taddei at TaMaRa's  Laboratory (Hôpital Necker, Paris) who have
 infected
the bacteria $\emph{E. Coli}$ with a lysogen
bacteriophage M13 \cite{tad}.
In  Section \ref{wcont}, we consider this model where a cell receive parasites only by 
inheritance from its mother cell. We  determine when 
 the number of
infected  cells becomes negligible compared to the number of cells
when the generation tends to infinity.\\

 \underline{II State dependent contamination}
In each generation, each cell may be contaminated  by a random
 number of parasites which also multiply randomly and are shared randomly between the two daughter
cells. This contamination  depends only on whether  the cell 
already contains parasites or not. \\

More formally, if a cell $\i$ contains $x$ parasites,
the contamination brings   $Y^{(0)}_x$ parasites to the first daughter cell of $\i$ and
$Y^{(1)}_x$ to the second one, where
$$\forall x\geq 1, \quad Y_1:\stackrel{d}{=}Y^{(0)}_x\stackrel{d}{=}Y^{(1)}_x, \qquad Y_0:\stackrel{d}{=}Y^{(0)}_0\stackrel{d}{=}Y^{(1)}_0.$$
Moreover we assume  that contamination satisfies
\be
\label{cond}
0<\P(Y_0=0)<1, \qquad 0<\P(Y_1=0),
\ee
which means that each  non-infected cell may be contaminated with a positive probability but    the cells are not   contaminated with probability  one. \\

This model   contains the case  when the contamination is independent
of the number of parasites in the cell ($Y_0$ and $Y_1$ are identically distributed). It also takes into account the case when only non infected cells can be contaminated ($Y_1=0$ a.s.) and the case when infected cells are 'weaker' and parasites contaminate them easier ($Y_1\geq Y_0$ a.s.). For biological and technical reasons, we dot make $Y_x$ depend on $x\geq 1$. But the results given here could be generalized to the case when the contamination depends on the number of parasites $x$ inside the cells soon as $x$ is less than some fixed constant.\\

\underline{III Cell division with parasite infection and contamination} We describe now the whole model.
We start with a single cell  with $k$ parasites and denote by $\P_k$ the associated probability.
Unless otherwise  specified, we assume $k=0$. \\

For every cell $\b{i}\in\TT$, conditionally on $Z_{\mathbf{i}}=x$ and $\f_{\i}=\g$, the numbers
of parasites $(Z_{\mathbf{i}0},
Z_{\mathbf{i}1})$ of its two daughter cells   is distributed as
$$
\sum_{k=1}^x (X^{(0)}_k(\i), X^{(1)}_k(\i)) + (Y^{(0)}_x (\i),Y^{(1)}_x(\i) ),
$$
where \\

(i)  $(X_k^{(0)}(\i) , X_k^{(1)}(\i) )_{k\geq 1}$ is an i.i.d. sequence
with common  couple p.g.f  $\g$.  \\

(ii) $(Y^{(0)}_x(\i),Y^{(1)}_x(\i))$ is independent
of $(X_k^{(0)}(\i) , X_k^{(1)}(\i))_{k\geq 1}$.  \\

Moreover, $\big((X_k^{(0)}(\i), X_k^{(1)}(\i) )_{k\geq 1}, \ (Y^{(0)}_x(\i),Y^{(1)}_x(\i))_{x\geq 0}\big)$ are i.i.d. for $\i \in \TT.$

\begin{fig} Cell division with multiplication of parasites, random sharing and contamination. Each parasite gives birth to a random number of light parasites and dark parasites.  Light parasites go into the first daughter cell, dark parasites go into
the second daughter cell and square parasites contaminate the cells from outside the cell population. But light/ dark/ square  parasites then behave in the same way.
\begin{center}
\includegraphics[scale=0.65]{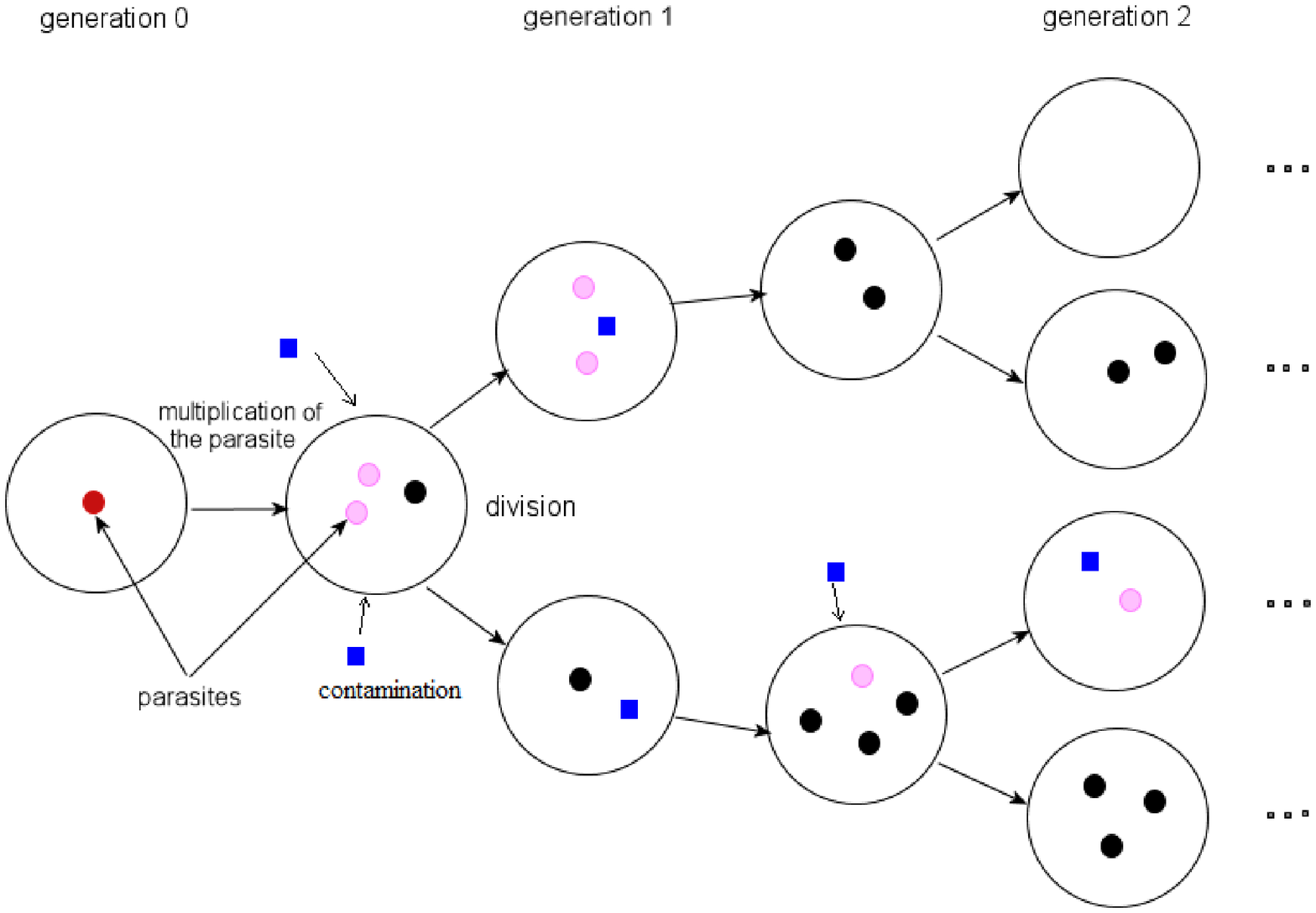}
\end{center}
\end{fig}

This model is a  Markov chain indexed by a tree. This subject has
been studied in  the literature (see e.g. \cite{atarbre, atarbres, benj}) in the symmetric
independent case. That is, $\forall (\b{i},k) \in \TT\times \N$,
$$ \P((Z_{\b{i}0}, Z_{\b{i}1})=(k_0,k_1)\mid Z_{\b{i}}=k)=\P(Z_{\b{i}0}=k_0\mid
Z_{\b{i}}=k)\P(Z_{\b{i}0}=k_1\mid Z_{\b{i}}=k).$$
But this identity does  not hold here since we are interested in unequal sharing of parasites. Guyon
\cite{guyon} proves limit theorems for  a Markov chain indexed by a binary
tree where  asymmetry and dependence are allowed.  His theorem is the key argument to prove the
convergence of asymptotic proportions of cells with a given number of parasites here.
Indeed, contamination ensures that the process which counts the number of parasites along  the random walk 
on the binary tree of the cell population is ergodic and non trivial (see
Section \ref{randline}). This
is the fundamental assumption to use Guyon's law of large numbers. Let us then introduce more precisely this process which gives the number of parasites in a random cell line.  \\

Let $(a_i)_{i \in \N}$ be an i.i.d. sequence  independent of
$(Z_{\b{i}})_{\b{i} \in  \TT}$ such that
\be \label{iid}
\P(a_1=0)=\P(a_1=1)=1/2. \ee
Denote by $f^{(0)}$ (resp $f^{(1)}$) the random p.g.f which gives the law of the
 size of the offspring of a parasite which goes in the first daughter cell (resp. in the second daughter cell): 
 $$f^{(0)}(s):=\f(s,1) \quad \t{a.s.}, \qquad f^{(1)}(t):=\f(1,t) \quad \t{a.s.}, \qquad (s,t\in [0,1]).$$
Let $f$ be the mixed generating function
of $f^{(0)}$ and $f^{(1)}$, i.e.
$$\P( f \in \d g )= \frac{\P(f^{(0)} \in  \d g) + \P(f^{(1)} \in \d g)}{2}.$$
Then
$(Z_n)_{n\in\N}=(Z_{(a_1,a_2,..a_n)})_{n \in \N}$ is a Branching
Process in Random Environment with immigration depending on the state is zero or not: the reproduction law is given by its p.g.f  $f$, the immigration law in zero is distributed as $Y_0$, and  the immigration law in $k\geq 1$ as   $Y_1$.  Thus, we first need to prove asymptotic results for this process.

\section{Main results}

Galton Watson processes with immigration are well  known (see e.g. \cite{asmu, RuL}). If the process is subcritical and the expectation of the logarithm of the immigration is finite, then it converges in distribution to a finite random variable. Otherwise it tends to infinity in probability. Key \cite{IBPRE} has obtained the analogue  result for  Branching Processes in Random Environment  with Immigration (IBPRE), in the subcritical case,  with finite expectation of the logarithm. Actually he states results for multitype IBPRE,  which have been complemented by Roitershtein \cite{Roit} who obtained a strong law of large numbers and a central limit theorem for the partial sum.\\

In Section \ref{IBPRE}, we give the asymptotic behavior of  IBPRE in the different cases, which means that we also consider the critical or supercritical case and the case when the expectation of the logarithm of the immigration is infinite.   To get these results, we use   some general statements  on Markov processes (Section \ref{MARK}), classical arguments for Galton Watson process with immigration (see
 \cite{RuL}, which was inspired from \cite{asmu}) and the tail of the time when IBPRE  returns to $0$ in the subcritical case, which is proved in  \cite{IBPRE}. \\

We can then state results about branching processes in random environment  $(Z_n)_{n\in\N}$ with immigration depending on the state is zero or not (Section \ref{randline}) using coupling arguments and Section \ref{MARK}.  This process gives the number of parasites along a random cell line. Recalling that immigration in state zero is distributed as $Y_0$ and immigration in state $k\geq 1$ is distributed as 
$Y_1$,  we prove the following expected result. \\

\begin{Thm*} \label{cvrandline}  (i) If $\E\big(\log(f'(1))\big)<0$ and
$\max(\E(\log^+(Y_i)): i=0,1)<\infty$, then there exists a finite r.v. $Z_{\infty}$ such that
for every $k \in \N$,
$Z_n$ starting from $k$ converges in distribution to  $Z_{\infty}$ as $n\rightarrow \infty$.

(ii) If $\E\big(\log(f'(1))\big)\geq 0$ or
$\max(\E(\log^+(Y_i)): i=0,1)=\infty$, then
$Z_n$ converges
in probability to $\infty$ as $n\rightarrow \infty$.
\end{Thm*}
 
With additional assumptions, we provide in Section \ref{randline} an estimate of the rate of  convergence of
$(Z_n)_{n\in\N}$ depending on the initial state. \\

Then, in  Section \ref{prop}, we  prove asymptotic results on the population of cells in generation $n$ as $n\rightarrow \infty$.

 First, we consider the case when  there is no contamination:  $Y_0=Y_1=0$ a.s.  We determine when the organism recovers, meaning that the number
of infected cells   becomes negligible compared to the total
number of cells. As stated in Proposition \ref{propsscont}, the recovery occurs a.s. iff
 $\E(\log(f'(1)))\leq 0$. Thus, we generalize results of Section 3 in \cite{vbk} to random environment. Again, for any reproduction rate of parasites, we can find a necessary and sufficient condition on sharing of their offspring so that the organism recovers a.s.    \\

As explained in introduction, a natural example is  the random binomial repartition of parasites. If the reproduction of parasites is given by the r.v. $Z$ 
and the random parameter   of the binomial repartition is given $P \in [0,1]$, the a.s. recovery criterion becomes
$$\log(\E(Z))\leq \E(\log(1/P)).$$
$\newline$

Second we take into account the contamination by parasites from outside the cell population with assumptions $(\ref{cond})$. We  focus on proportions of cells in generation $n$ with a given number of parasites:
$$F_k(n):= \frac{ \# \{ \mathbf{i}\in \GG_n :
Z_{\mathbf{i}}=k\}}{2^n} \qquad (k\in\N).$$
Using \cite{guyon} and the theorem above, we prove  the
 following  law of large numbers. 
\begin{Thm*}
If $\E(\log(f'(1)))<0$ and $\max(\E(\log^+(Y_i)): i=0,1)<\infty$,
then for every $k\in\N$, $F_k(n)$ converges in probability  to
a deterministic number   $f_k$ as $n\rightarrow
\infty$, such that $f_0>0$ and $\sum_{k=0}^{\infty} f_k=1$.  \\
Otherwise, for every $k\in\N$, $F_k(n)$ converges in probability  to
$0$ as $n\rightarrow\infty$.
\end{Thm*}

Finally, in Section \ref{nbparasit}, we give the asymptotic behavior of the total number of parasites in generation $n$ in the case when the growth of parasites follows a Galton Watson process and the contamination does not depend on the state of cell.

\section{Preliminaries}

We recall first some results about Branching Processes in Random Environment (BPRE) and then about Markov chains, which will be both useful to study BPRE with immigration $(Z_n)_{n\in\N}$. Recall that we denote by $k$ the initial number of parasites and by $\P_k$ the probability associated with. \\

\subsection{Branching Processes in Random Environment (BPRE)}
\label{BPRE}
 We consider here a BPRE  $(Z_n)_{n\in\N}$ specified by
a sequence of i.i.d. generating functions $(f_n)_{n\in\N}$ distributed
as $f$ \cite{at, atr, bpre}. More precisely, conditionally
on the environment
 $(f_n)_{n\in\N}$, particles at generation $n$ reproduce
independently of each other and their offspring has generating function $f_n$. Then $Z_n$
is the number of particles at generation $n$
and $Z_{n+1}$ is the sum of $Z_n$ independent random variables with generating function $f_n$.
That is, for every $n\in\N$,
$$\E\big(s^{Z_{n+1}}\vert Z_0,\dots,Z_n; \ f_0,\dots,f_n\big)=f_n(s)^{Z_n} \qquad (0\leq s\leq 1).$$
Thus, denoting by $F_n:=f_{0}\circ \cdots\circ f_{n-1}$, we have for every $k\in\N$,
$$\E_k(s^{Z_{n+1}}\ \vert \ \ f_0, ...,  f_n)=\E(s^{Z_{n+1}} \mid Z_0=k, \ f_0, ...,  f_n)=F_n(s)^k \qquad (0\leq s\leq 1).$$

When the environments are deterministic
(i.e. $f$ is a deterministic generating function), this process is the Galton Watson
process  with reproduction law $N$, where $f$ is the generating function of $N$ . \\

The process $(Z_n)_{n\in\N}$ is called subcritical, critical or supercritical if
$$\E\big(\log(f'(1))\big)$$
is negative, zero or positive  respectively. This process becomes extinct a.s.:
$$\P(\exists n \in \N: Z_n=0)=1$$
iff it is subcritical or critical \cite{at} (see \cite{bpree} for finer results).
$\newline$

In the critical case, we make the following integrability assumption:
  $$0<\E(\log(f_0'(1))^2)<\infty, \quad  \E\big( [1+ \log(f_0'(1))]f_0''(1)/2f_0'(1) \big)<\infty,$$
so that  there exist $0<c_1<c_2<\infty$ such that for every $n\in\N$ (see \cite{koz})
 \be
\label{eqvltcrtq}
c_1/\sqrt{n}\leq \P(Z_n>0)\leq c_2/\sqrt{n}.
\ee
See \cite{agkv} for more general result in the critical case.\\

\subsection{Markov chains}
\label{MARK}
We consider now a Markov chain $(Z_n)_{n\in\N}$ taking values in $\N$ and introduce   the first time  $T_0$ when $(Z_n)_{n\in\N}$
visits $0$ after time $0$:
$$T_0:=\inf\{i >0: Z_i=0\}.$$
Denote by
$$u_n:=\P_0( Z_n=0), \qquad  u_{\infty}:=1/\E_0(T_0) \qquad (1/\infty=0).$$
By now, we assume $0<\P_0(Z_1=0)<1$ and we give the asymptotic behavior of $(Z_n)_{n\in\N}$. The first part of (i) is the classical ergodic property  for  an aperiodic positive recurrent Markov chain and we provide an estimate of the  speed of convergence  depending on the initial state. Then (ii) gives the null recurrent   case, which is also a classical result. $\newline$
\begin{Lem} \label{mark}
 (i) If for every $k\in\N$,  $\E_k(T_0)<\infty$, then
$Z_n$ starting from $k$  converges in distribution to a finite random variable $Z_{\infty}$, which does not depend
on $k$ and verifies
$$\P(Z_{\infty}=0)>0.$$
Moreover  there exists $A>0$ such that for all $n,k \in\N$,
\bea
&&\sum_{l\in\N} \vert \P_k(Z_n=l)- \P(Z_{\infty}=l) \vert \nonumber \\
\label{ctrl}
&&\leq   A\big[\sup_{n/2\leq l\leq n}\{\vert u_{l}- u_{\infty}\vert \} +
\E_0(T_0\ind_{T_0>n/4})+\E_k(T_0\ind_{T_0>n/4})\big].
\eea

(ii) If $\E_0(T_0)=\infty$ and for every $l\in\N$, $\P_l(T_0<\infty)>0$, then for every $k\in\N$, $Z_n\rightarrow \infty$ in $\P_k$-probability   as $n\rightarrow \infty$.
\end{Lem}
$\newline$
\begin{proof}[Proof of (i)]
 First,  note that by the Markov property, for every $n\in\N$,
\bea
&&\vert \P_k(Z_n=0)-u_{\infty} \vert \nonumber \\ 
&= & \vert \sum_{j=1}^n \P_k(T_0=j)\P_0(Z_{n-j}=0) - u_{\infty}\vert \nonumber \\
\label{majtmps}
&\leq & \sum_{j=1}^n \P_k(T_0=j) \vert u_{n-j} - u_{\infty}\vert +u_{\infty}\P_k(T_0>n).
\eea
On the event $\{T_0\leq n\}$, define  $R_n$  as the last passage time of $(Z_n)_{n \in \N}$ by $0$ before time $n$:
$$R_n:=\sup\{ i\leq n: Z_i=0\}.$$
For all $0\leq i\leq n$ and $l\in\N$, by the  Markov property,
\bea
\P_k(Z_n=l) 
&=&\P_k(T_0>n, \ Z_n=l) +\sum_{i=0}^n  \P_k(T_0\leq n, \ R_n=n-i,  \ Z_n=l)  \nonumber \\
&=& \P_k(T_0>n, \ Z_n=l) +\sum_{i=0}^n \P_k(Z_{n-i}=0) \P_0(Z_i=l, \ T_0>i). \nonumber
\eea
Define now
$$\alpha _ l:= u_{\infty}
\sum_{i=0}^{\infty}   \P_0(Z_i=l, \ T_0>i).$$
We then have 
\Bea
\vert \P_k(Z_n=l)- \alpha_l \vert &\leq & \P_k(T_0>n, \ Z_n=l)+u_{\infty}
\sum_{i=n+1}^{\infty}  \P(Z_i=l, \ T_0>i) \\
&& + \sum_{i=0}^n \P(Z_i=l, \ T_0>i)\big\vert u_{\infty}- \P_k(Z_{n-i}=0)\big\vert.
\Eea
Summing over $l$  leads to
\bea
\sum_{l\in\N} \vert \P_k(Z_n=l)- \alpha_l \vert
& \leq & \P_k(T_0>n)+u_{\infty} \E_0(T_0\ind_{T_0>n+1}) \nonumber\\
&& \quad +  \sum_{i=0}^{n}    \P(T_0>i)
\big\vert u_{\infty}- \P_k(Z_{n-i}=0)\big\vert. \label{one}
\eea
Moreover  using  (\ref{majtmps}), we have for all $0\leq n_0\leq n$,
\bea
&&\sum_{i=0}^{n}  \P(T_0>i)
\big\vert u_{\infty}- \P_k(Z_{n-i}=0) \big\vert \nonumber \\
&\leq & \sum_{i=0}^{n}  \P_0( T_0>i)\left[\sum_{j=1}^{n-i} \P_k(T_0=j) \vert u_{n-i-j} - u_{\infty}\vert +u_{\infty}\P_k(T_0>n-i)\right]\nonumber \\
&\leq &  \sum_{i=0}^{n}\P_0( T_0>i)  \sum_{j=1}^{n-i} \P_k(T_0=j) \vert u_{n-i-j} - 
u_{\infty}\vert+ u_{\infty} \sum_{i=0}^{n}  \P_0( T_0>i)\P_k(T_0>n-i). \label{two}
\eea
Finally, denoting by $M:=\sup_{n\in\N} \{\vert u_n - u_{\infty}\vert\}$,
\bea
&& \sum_{i=0}^{n}\P_0( T_0>i)  \sum_{j=1}^{n-i} \P_k(T_0=j) \vert u_{n-i-j} - u_{\infty}\vert \nonumber \\
& & \qquad \leq \sup_{n_0\leq l\leq n}\{\vert u_{l}- u_{\infty}\vert \}\sum_{i=0}^{n}   \P_0(T_0>i) \sum_{j=1}^{n-i} \P_k(T_0=j) 1_{n-i-j\geq n_0} \nonumber\\
&& \qquad  \quad +M \sum_{i=0}^{n}   \P_0(T_0>i) \sum_{j=1}^{n-i} \P_k(T_0=j) 1_{n-i-j< n_0} \nonumber\\
& &  \qquad \leq \sup_{n_0\leq l\leq n}\{\vert u_{l}- u_{\infty}\vert \}\sum_{i=0}^{n}   \P_0(T_0>i) \sum_{j=1}^{n-i} \P_k(T_0=j) \qquad  \quad \qquad  \quad \nonumber\\
&& \qquad  \quad + M\sum_{i=0}^{n-n_0}   \P_0(T_0>i)\P_k(T_0> n-n_0-i). \label{three}
\eea
Combining (\ref{one}), (\ref{two}) and (\ref{three})  and
using  that
\Bea
&& \sum_{i=0}^{n}   \P_0(T_0>i)  \P_k(T_0>n-i)\leq \E_0(T_0 \ind_{T_0> n/2})+\E_k(T_0 \ind_{T_0> n/2}), \\
&& \sum_{i=0}^{n-n_0}   \P_0(T_0>i)\P_k(T_0\geq n-n_0-i) \leq \E_0(T_0 \ind_{T_0> (n-n_0)/2})+\E_k(T_0 \ind_{T_0>n-n_0)/2}),
\Eea
we get,  for all $0\leq n_0\leq n$,
\bea \sum_{l\in\N} \vert \P_k(Z_n=l)- \alpha_l \vert
&\leq  & \P_k(T_0>n)+u_{\infty} \E_0(T_0\ind_{T_0>n+1})+\sup_{n_0\leq l\leq n}\{\vert u_{l}- u_{\infty}\vert \} \E_0(T_0) \nonumber \\
&& \label{nine} \quad
 +[u_{\infty}+M][\E_0(T_0 \ind_{T_0> (n-n_0)/2})+\E_k(T_0 \ind_{T_0> (n-n_0)/2})].
\eea
As $\P_0(Z_1=0)>0$,  by the renewal theorem \cite{Feller}, 
$u_n\stackrel{n\rightarrow \infty}{\longrightarrow } u_{\infty}.$ 
Adding that
$\E_k(T_0)<\infty$ and $\E_0(T_0)<\infty$ ensures that
$$\sum_{l\in\N} \vert \P_k(Z_n=l)- \alpha_l \vert  \stackrel{n\rightarrow \infty}{\longrightarrow }0,$$
which proves that $Z_n$ starting from $k$ converges in distribution  to a r.v. $Z_{\infty}$ which does not depend on $k$.

The inequality of $(i)$ is obtained by letting $n_0=n/2$ in (\ref{nine}).
\end{proof}
$\newline$
\begin{proof}[Proof of (ii)]
If $\E_0(T_0)=\infty$, then by the renewal theorem again \cite{Feller},
$$u_n\stackrel{n\rightarrow \infty}{\longrightarrow} 0.$$
So
$$D_n=\inf\{k-n: k \geq n, Z_k=0\} \stackrel{n\rightarrow \infty}{\longrightarrow} \infty, \qquad  \t{in probability}.$$
 Assume that there exist $l\in\N$,  $\epsilon>0$ and an increasing sequence of integers $(u_n)_{n\in \N}$
 such that
 $$\P_k(Z_{u_n}=l)\geq \epsilon.$$
 As $\P_l(T_0<\infty)>0$ by hypothesis, there exists $N>0$ such that
 $$\P_l(T_0=N)>0.$$
 Thus, by the Markov property,
 $$\P_k(Z_{u_n+K}=0)\geq \P_k(Z_{u_n}=l)\P_l(T_0=N)\geq \epsilon \P_l(T_0=N).$$
Then, for all $n \in \N$,
$$\P_k(D_{u_n}\leq N)\geq   \epsilon \P_l(T_0=N)>0,$$
which is in contradiction with the fact that $D_n\rightarrow \infty$ in $\P_k$ as $n\rightarrow \infty$. Then,
$\P_k(Z_n=l)\rightarrow 0$ as $n\rightarrow \infty$.
\end{proof}
$\newline$

\section{Branching processes in random environment with immigration
(IBPRE)} \label{IBPRE}
 We consider here a BPRE $(Z_n)_{n\in\N}$ whose
reproduction law is given  by the random p.g.f $f$ and  we add at each
generation $n+1$ a random number of immigrants $Y_n$ independent and
identically distributed as a r.v $Y$ such that
$$\P(Y=0)>0.$$
More precisely, for every $n\in \N$,
\be
\label{eqimmigrat}
Z_{n+1}=Y_{n}+ \sum_{i=1}^{Z_n} X_i,
\ee
where $(X_i)_{i\in\N}$, $Y_n$ and $Z_n$ are independent and
conditionally on $f_n=g$, the $(X_i)_{i\in\N}$ are i.i.d. with common probability
generating function $g$. \\

Note that if the contamination  does not dependent on the
fact that this cell is already infected or not (i.e.
$Y_0$ and $Y_1$ are identically distributed), then the number of parasites in a random
cell line defined in Introduction is a IBPRE whose reproduction law given
by $f$ and immigration by $Y\stackrel{d}{=}Y_0\stackrel{d}{=}Y_1$. \\

We give now the asymptotic behavior of this process. These results are classical for the Galton Watson process with immigration \cite{asmu, RuL}. We follow the same method in the case of random environment for the subcritical and supercritical cases. We   give in (ii)
the tail of the time 
$$T_0=\inf\{n>0: Z_n=0\}$$ when the process returns to $0$ in the subcritical case, which is proved in  \cite{IBPRE} and we use Section \ref{MARK} for the critical case. 
$\newline$

\begin{Pte} \label{ibpre} (i) If $\E\big(\log(f'(1))\big)<0$ and  $\E(\log^+(Y))<\infty$, then $Z_n$ converges in distribution to a finite random variable as $n\rightarrow \infty$ and $\lim_{n\rightarrow \infty} \P(Z_n=0)>0$. \\
Otherwise
$Z_n\rightarrow \infty$ as  $n\rightarrow \infty$.\\

(ii)  If
$\E\big(\log(f'(1))\big)<0$ and there exists $q>0$ such that
$\E(Y^q)<\infty$, then there exist $c,d>0$ such that for every $n\in\N$,
$$\P(T_0>n)\leq ce^{-dn}.$$

(iii) Assume  $\E(f'(1)^{-1})<1$ and  $\E(\log^+(Y))<\infty$, then there exists a  finite
r.v. $W$ such that
$$[\Pi_{i=0}^{n-1} f_i'(1)]^{-1}Z_n
\stackrel{n\rightarrow \infty}{\longrightarrow } W, \quad \t{in} \ \P.$$
\end{Pte}
$\newline$

Note also that by the Borel-Cantelli lemma,  if $\E(\log^+(Y_1))=\infty$, then for every $c>1$,
$$\limsup_{n\rightarrow \infty} c^{-n}Z_n = \infty \quad \t{a.s.}$$
since $Z_n\geq Y_n$ a.s. 
Moreover the proof of Section \ref{randline} provides an other approach to prove that $(Z_n)_{n\in\N}$
tends to  $\infty$ if $\E(\log^+(Y))=\infty$. \\

\begin{proof}[Proof of (i) and (ii) in the subcritical case: $\E\big(\emph{log}(f'(1))\big)<0$]
The subcritical case  with
assumption  $\E(\log^+(Y))<\infty$ is handled in \cite{IBPRE}: First
part of (i) is Theorem 3.3 and  (ii) is a consequence
of Theorem 4.2 of \cite{IBPRE}. \\

We focus now on the case  $\E(\log^+(Y))=\infty$ and prove that $Z_n$ converges in probability to $\infty$.  The proof
is close to the Galton Watson case (see \cite{asmu} or \cite{RuL}). First, by Borel-Cantelli lemma,
$$ \limsup_{k\rightarrow \infty} \log^+(Y_k)/k =\infty \qquad a.s.$$
Then, for every $c\in (0,1)$,
\be
\label{bc}
\limsup_{k\rightarrow \infty} c^kY_k=\infty \qquad a.s.
\ee
$\t{} \qquad$  Note that
$$Z_n=\sum_{k=0}^{n-1} Z_{k,n},$$
where $Z_{k,n}$ is the number of descendants in generation $n$ of immigrants in generation $n-k$. Thus, denoting by
$Y_{k,n}$ the number of immigrants in generation $n-k$ and $X_i(k,n)$ the number of descendants in generation $n$ of immigrant $i$ in generation $n-k$, we have
$$Z_n=\sum_{k=0}^{n-1} \sum_{i=1}^{Y_{k,n}} X_i(k,n).$$
This sum increases stochastically  as $n$ tends to infinity and converges in distribution to
$$Z_{\infty}=\sum_{k=0}^{\infty} \sum_{i=1}^{Y_{k}} X_i(k),$$
where conditionally on $(f_i: i \in \N)$, $(X_i(k): \ i \in \N, \ k \in \N)$ are independent  and  the probability generating function of $X_i(k)$ is equal to $f_{k-1}\circ ...\circ f_0$. Roughly speaking, $X_i(k)$ is the contribution of immigrant $i$ which arrives $k$ generations before 'final time' $\infty$. The integer $X_i(k)$ is   the population in generation $k$ of a BPRE  without immigration starting from $1$. \\

Assume now that $Z_{\infty}<\infty$ with a positive probability. As $(X_i(k): \ k \in \N, \ 1\leq i\leq Y_k)$ are integers, then conditionally on $Z_{\infty}<\infty$, only a finite number of them are positive. Thus, by Borel-Cantelli lemma,  conditionally
 on
$(Z_{\infty}<\infty, \ Y_k:  k \in \N, \  f_i: i \in \N)$,
$$\sum_{k=0}^{\infty} Y_k\P(X_1(k) > 0) < \infty \qquad \t{a.s.}$$
Moreover, by convexity, for all $g$ p.g.f and  $s\in[0,1]$,
$$\frac{1-g(s)}{1-s}=\frac{g(1)-g(s)}{1-s}\geq \frac{g(1)-g(0)}{1-0}=1-g(0),  \qquad (0\leq s\leq 1).$$
Then $1-g(s)\geq (1-g(0))(1-s)$  and  by induction, we have for every $k\in\N$,
\Bea \P(X_1(k)> 0 \ \vert \ f_i: i \in \N)&=&1-f_{k-1}\circ ...\circ f_0(0) \\
& \geq & \Pi_{i=0}^{k-1} (1-f_i(0)) \\
& = & \exp(S_k),
\Eea
where $S_k:=\sum_{i=0}^{k-1} \log(1-f_i(0))$. Thus, conditionally
 on
$(Z_{\infty}<\infty, \ Y_k:  k \in \N, \  f_i: i \in \N)$,
$$\sum_{k=0}^{\infty} Y_k \exp(S_k)  < \infty \qquad \t{a.s.}$$
Thus, on the event $\{Z_{\infty}<\infty\}$ which has a positive probability, we get
$$\sum_{k=0}^{\infty} Y_k \exp(S_k)  < \infty \ \t{a.s.}$$
Moreover $S_n$ is a random walk with negative drift $\E(\log(1-f_0(0)))$. So
 letting $\alpha<\E(\log(1-f_0(1)))$, $\P(S_n<\alpha n)$ decreases exponentially by classical large deviation results. Then by Borel-Cantelli lemma, $S_n$ is less than $\alpha n$ for a finite number of $n$, and
$$L:=\inf_{n \in \N} \{S_n-\alpha n\}>-\infty \qquad \t{a.s.}$$
Using that for every $k\in\N$, $S_k\geq \alpha k+L$ a.s., we get
$$\sum_{k=0}^{\infty} \exp(\alpha k) Y_k   < \infty,$$
with positive probability. This is in contradiction with  $(\ref{bc})$. Then $Z_{\infty}=\infty$ a.s. and $Z_n$ converges in probability to $\infty$ as $n\rightarrow \infty$.
\end{proof}
$\newline$
\begin{proof}[Proof of (i) in the critical and supercritical case: $\E\big(\emph{log}(f'(1))\big)\geq  0$]
First, we focus on the critical case. Recall that $ T_0=\inf\{i>0: Z_i=0\}$ and consider
$(\bar{Z}_n)_{n\in\N}$ the BPRE associated with
$(Z_n)_{n\in\N}$, that is the critical BPRE with reproduction
law $f$ and no immigration. Thanks to (\ref{eqvltcrtq}), there exists $c_1>0$ such that for ever $n\in\N$,
$$\P_1(\bar{Z}_n > 0)\geq c_1/\sqrt{n}.$$
Adding that
$$\P_1(T_0>n)=\P_1(Z_n>0)\geq \P_1(\bar{Z}_n>0),$$
ensures that
$$\E_1(T_0)=\infty.$$
Then $\E_0(T_0)=\infty$ since IBPRE $(Z_n)_{n\in\N}$ starting from $1$ is stochastically larger than $(Z_n)_{n\in\N}$ starting from $0$. Moreover $\forall k \in \N,  \ \P_k(T_0<\infty)>0$, since $\P_k(\bar{T}_0<\infty)=1$ and
$ \P(Y=0)>0$. Then
Lemma \ref{mark} (ii) ensures that $Z_n\rightarrow \infty$ in $\P$ as $n\rightarrow \infty$. \\

For the supercritical case, follow the proof in the critical case (or use the result  with a coupling argument) to get 
 that $Z_n\rightarrow \infty$ in probability  as $n\rightarrow \infty$
\end{proof}
$\newline$
\begin{proof}[Proof of  (iii)]
We   follow again \cite{RuL}.  If $\E(\log^+(Y))<\infty$,  by Borel-Cantelli Lemma
$$\limsup_{k\rightarrow \infty} \log^+(Y_k)/k=0.$$
Then for every $c>1$,
\be
\label{bcd}
\sum_{k=0}^{\infty} c^{-k}Y_k<\infty \qquad \t{a.s.}
\ee
Define 
$$P_n:=[\Pi_{i=0}^{n-1}f_i'(1)]^{-1},$$
and denote by $\mathcal{F}_n$ the $\sigma$-field generated by $(Z_i: 0\leq i \leq n)$, $(P_i: 0\leq i\leq n)$
and $(Y_k: k\in \N)$. 
 Then using $(\ref{eqimmigrat})$, we have
\Bea \E(P_{n+1}Z_{n+1} \ \vert \
\mathcal{F}_n)&=&
\E(P_{n+1}[\sum_{i=1}^{Z_n}X_i + Y_n] \ \vert \ \mathcal{F}_n) \\
&=& P_n\E( f'_{n}(1)^{-1} \sum_{i=1}^{Z_n} X_i \ \vert \ \mathcal{F}_n) +P_n \E(f'(1)^{-1})Y_n \\
&=& P_n\E( f'_{n}(1)^{-1} Z_n \E(X_1 \ \vert \ f_n) \ \vert \ \mathcal{F}_n) +P_n \E(f'(1)^{-1})Y_n \\
&=&P_nZ_n+P_n\E (f'(1)^{-1})Y_n.
\Eea
So $P_nZ_n$ is a submartingale. Moreover
\Bea
\E(P_nZ_n \ \vert \  \mathcal{F}_0)&=& Z_0+\sum_{i=0}^{n-1}\E(f'(1)^{-1})^{i+1}Y_i.
\Eea
By (\ref{bcd}), if $\E(f'(1)^{-1})<1$, $P_nZ_n$ has bounded expectations and then converges a.s. to a finite r.v.
\end{proof}
$\newline$

\section{Ergodicity and convergence for a random cell line}
\label{randline}
Recall that  $(Z_n)_{n\in\N}$ defined in Introduction is the number of
parasites in a random cell line. The Markov chain  $(Z_n)_{n\in\N}$ is a BPRE 
with state dependent immigration. The reproduction law is given by the p.g.f $f$, immigration  in state
 $0$ is distributed as $Y_0$ and immigration in state $k\geq 1$ is distributed as $Y_1$. More precisely, for every 
$n\in \N$, conditionally on $Z_n=x$,
$$Z_{n+1}=Y_{x}^{(n)}+ \sum_{i=1}^{x} X_i^{(n)},   \qquad \t{where}$$
(i)  $(X_i^{(n)})_{i\in\N}$ and $Y_x^{(n)}$ are independent. \\
(ii) Conditionally on $f_n=g$, the $(X_i^{(n)})_{i\in\N}$ are i.i.d. with common probability
generating function $g$. \\
(iii) For all  $x\geq 1$ and $n\in\N$, $Y_x^{(n)}\stackrel{d}{=}
Y_1.$   \\

We have the following results, which generalize  those of
the previous section to the case when immigration depends on whether the state is zero or not.  $\newline$
\begin{Thm} \label{cvrandline}  (i) If $\E\big(\log(f'(1))\big)<0$ and
$\max(\E(\log^+(Y_i)): i=0,1)<\infty$, then there exists a finite random variable $Z_{\infty}$ such that
for every $k \in \N$,
$Z_n$ starting from $k$ converges in distribution to $Z_{\infty}$ as $n\rightarrow \infty$.

Moreover, if there exists $q>0$ such that $\max (\E(Y^q_i): i=0,1)<\infty$, then  for every $\epsilon>0$, there exist
$0<r<1$  and $C>0$
 such that for all  $n\in\N$ and  $k\in\N$,
$$ \sum_{l=0}^{\infty} \vert \P_k(Z_n=l) - \P(Z_{\infty}=l) \vert \leq C k^{\epsilon} r^n.
$$

(ii) If $\E\big(\log(f'(1))\big)\geq 0$ or
$\max(\E(\log^+(Y_i)): i=0,1)=\infty$,
$Z_n$ converges
in probability to infinity as $n\rightarrow \infty$.
\end{Thm}
$\newline$
Note again that by Borel-Cantelli lemma, if  $\E(\log^+(Y_1))=\infty$, then for every $c>1$,
$$\limsup_{n\rightarrow \infty} c^{-n}Z_n = \infty \quad \t{a.s.},$$
since $Z_n\geq Y_n$ a.s. \\

 The proof of (ii) in the critical or supercritical case ($\E\big(\log(f'(1))\big)\geq 0$) is directly derived from
Proposition \ref{ibpre} and we focus now on the subcritical case:
$$\E\big(\log(f'(1))\big)<0.$$
Recall that  $T_0$ is the first time after $0$  when $(Z_n)_{n\in\N}$
visits $0$. Using IBPRE (see Section \ref{IBPRE}), we prove the following
result in the subcritical case.
$\newline$
\begin{Lem}
\label{lem} If  $ \max(\E(\log^+(Y_i)): i=0,1)<\infty$, then for
every $k\N$, $\P_k(T_0<\infty)=1$ and 
$$\sup_{n\in\N}\{\P_k(Z_n\geq l)\}\stackrel{l\rightarrow \infty}{\longrightarrow }0.$$
$\newline$
 Moreover if   there exists $q>0$ such that $\max (\E(Y^q_i)
: i=0,1)<\infty$, then for every $\epsilon>0$, there exist $r>0$ and
$C>0$ such that for  all $n \in \N$, $k\geq 1$: 
$$\P_0(T_0\geq n)\leq Cr^n, \quad 
\P_k(T_0\geq n)\leq  Ck^{\epsilon}r^n. 
$$
\end{Lem}
\begin{proof}
We couple $(Z_n)_{n\in\N}$ with an IBPRE $(\w{Z}_n)_{n\in\N}$ with
reproduction law given by the random p.g.f $f$ (such as
$(Z_n)_{n\in\N}$) and immigration  $Y$  defined by
$$Y:= \max (Y_0,  \ Y_1, \ \w{Y}),$$
where $Y_0,  \ Y_1,$ and $\w{Y}$ are independent and $\w{Y}$ is defined by
$$\P( \w{Y}= 0)=1/2; \quad  \forall n \in \N^*, \quad \P( \w{Y}= n)= \alpha n^{-1-\epsilon}, \quad \alpha:=[2\sum_{i=1}^{\infty} i^{-1-\epsilon}]^{-1}.$$
Thus immigration   $Y$ for $\w{Z}_n$ is stochastically larger than immigration for $Z_n$ (whereas reproduction law is the same), so that coupling gives
$$\forall n \in \N, \quad Z_n\leq \w{Z}_n \qquad \t{a.s.}$$
Moreover, $\w{Z}_n$ is still subcritical. Recalling that
$\min(\P(Y_i=0): i=0,1)>0$,  $\P( \w{Y}= 0)=1/2$, and that the
expectation of the logarithm of every  r.v. is finite, we have
$$\E(\log^+(Y))<\infty, \qquad \P(Y=0)>0.$$
Then Proposition \ref{ibpre} (i)  ensures that $\w{Z}_n$ converges
in distribution to a finite random
 variable, so that
$$\sup_{n\in\N}\{\P_k(Z_n\geq l)\}
\leq \sup_{n\in\N}\{\P_k(\w{Z}_n\geq
l)\}\stackrel{l\rightarrow \infty}{\longrightarrow}0 .$$
Proposition
\ref{ibpre} (i)  ensures also that for every $k\in \N$, $\lim_{n\rightarrow \infty}
\P_k(Z_n=0)>0$. Thus, for every
$k\in\N$, $\P_k(\w{T}_0<\infty)=1$ and then $\P_k(T_0<\infty)=1$. This completes the first part of the lemma. \\ \\

We assume now that there exists $q>0$ such that $\max (\E(Y_i^q): i=0,1)<\infty$. Moreover $\E(\w{Y}^{\epsilon/2})<\infty$, so letting $q'=\min(\epsilon/2, q)$,  we have
$$\E(Y^{q'})<\infty.$$
We can then apply Proposition \ref{ibpre} (ii) to IBPRE $(\w{Z}_n)_{n\in\N}$, so that there exist
$c,d>0$ such that for every $n\in\N$,
$$\P_0(\w{T_0} > n)\leq ce^{-dn}.$$
Recalling that  for all $k,n\in\N$,
$$\P_k(T_0\geq n)\leq \P_k(\w{T}_0\geq n),$$
we get $\P_0(T_0\geq n)\leq ce^{-dn}$. Moreover  for every $k\in\N$,
$$\P_0(\w{T_0} > n)\geq \P(Y\geq k)\P_k(\w{T_0} \geq  n).$$
By definition of $Y$,   there exists $\beta>0$ such that for every
$n\in\N$,
$$\P(Y\geq n)\geq \beta n^{-\epsilon}.$$
Using these inequalities gives 
$$
\P_k( T_0 \geq  n)\leq  \P_k(\w{T_0} \geq  n) \leq  \beta^{-1}k^{\epsilon}\P_0(\w{T_0} > n)
\leq  \beta^{-1}ck^{\epsilon}e^{-dn}.
$$
This completes the proof.
\end{proof}
$\newline$
\begin{proof}[Proof of Theorem \ref{cvrandline} (i) and (ii) in the subcritical case: $\E\big(\emph{log}(f'(1))\big)<0$.]
We split the proof into  $4$ cases: \\ \\
CASE 1: $\max(\E(\log^+(Y_i)) : i=0,1)<\infty$. \\ \\
CASE 2: There exists $q>0$ such that  $\max (\E(Y^q_i): i=0,1)<\infty$. \\ \\
CASE 3: $\E(\log^+(Y_1))=\infty$. \\ \\
CASE 4: $\E(\log^+(Y_0))=\infty$. \\ \\

First, note that $\P(Y_0=0)>0$ ensures that $\P_0(Z_1=0)>0$ and we can use results of Section \ref{MARK}. \\

CASE 1. In this case, by Lemma \ref{lem}, $(Z_n)_{n\in\N}$ is bounded in distribution:
$$\sup_{n\in\N}\{\P_0(Z_n\geq l)\}\stackrel{l\rightarrow \infty}{\longrightarrow }0.$$
If $\E_0(T_0)=\infty$, then $Z_n\rightarrow \infty$ in $\P_0$ by
Lemma \ref{mark} (ii), which is in contradiction
with the previous limit. \\
Then  $\E_0(T_0)<\infty$. We prove now that $\forall k\geq 1$,
$\E_k(T_0)<\infty$ by a coupling argument. Let $k\geq 1$ and change
only immigration to get a Markov process
 $(\w{Z}_n)_{n\in\N}$ which is larger than $(Z_n)_{n\in\N}$:
 $$\forall n\in\N, \qquad \w{Z}_n\geq Z_n \ \t{a.s.}$$
Its  immigrations $\w{Y}_0$ and $\w{Y}_1$ satisfy
$$\w{Y}_1\stackrel{d}{=} Y_1, \quad  \forall n\in \N, \ \P(\w{Y}_0\geq n)\geq \P(Y_0\geq n),$$
$$\P(\w{Y}_0\geq  k)>0, \quad  \max(\E(\log(\w{Y}_i): i=0,1)<\infty.$$
Then, we have again $\E_0(\w{T}_0)<\infty$, which entails that $\E_k(\w{T}_0)<\infty$ since  $\P(\w{Y}_0\geq  k)>0$. As for every $n\in\N$,
$\w{Z}_n\geq Z_n$ a.s., we have
$$\E_k(T_0)\leq \E_k(\w{T}_0)<\infty.$$
Then Lemma \ref{mark} (i) ensures that for every $k\in \N$,
$(Z_n)_{n\in\N}$ converges in distribution to a finite random variable $Z_{\infty}$, which does not depend
on $k$ and verifies $\P(Z_{\infty}=0)>0$. \\ 

CASE 2:  By Lemma \ref{mark} (i), we have
\be
\label{majsom}
\sum_{l\in\N} \vert \P_k(Z_n=l)- \P(Z_{\infty}=l) \vert
\leq   A\big[\sup_{n/2\leq l\leq n}\{\vert u_{l}- u_{\infty}\vert \} +
\E_0(T_0\ind_{T_0>n/4})+\E_k(T_0\ind_{T_0>n/4})\big].
\ee
Moreover by Lemma \ref{lem}, for every $\epsilon>0$, there exists $C>0$ such that
\be
\label{majun}
\P_k(T_0\geq n)\leq  Ck^{\epsilon}r^n, \qquad
\P_0(T_0\geq n)\leq Cr^n.
\ee
So for every $r'\in (r,1), \ \E_0(\exp( -\log(r) T_0))<\infty$. Then, by
Kendall renewal theorem \cite{kendall}, there  exists $\rho \in (0,1)$ and $c>0$ such that
for every $n\in\N$,
\be
\label{geom}
\vert u_n - u_{\infty} \vert \leq c\rho^{n}.
\ee
Finally, (\ref{majun})  ensures that there exists  $D>0$ such that for every $n\in\N$,
\Bea
\E_0(T_0\ind_{T_0>n/4})&\leq & Dnr^{n/4},\\
\E_k(T_0\ind_{T_0>n/4}) &\leq & Dnk^{\epsilon}r^{n/4} .
\Eea
Combining these two inequalities with $(\ref{majsom})$ and $(\ref{geom})$, we get
$$\sum_{l\in\N} \vert \P_k(Z_n=l)- \P(Z_{\infty}=l) \vert
\leq   A\big[c\rho^n+Dnr^{n/4}+ Dnk^{\epsilon}r^{n/4}\big],$$
which ends the proof in CASE 2. \\ \\

CASE 3. Change immigration of $(Z_n)_{n \in\N}$ to get an IBPRE
$(\w{Z}_n)_{n\in\N}$ whose immigration is distributed as $Y_1$ and
whose reproduction law is still given by $f$. Then Proposition
\ref{ibpre} (i) and $\E(\log^+(Y_1))=\infty$
ensures that $(\w{Z}_n)_{n\in\N}$ starting from $0$ tends in distribution to $\infty$. \\
Then Lemma \ref{mark} (i) entails  that $\E_0(\w{T}_0)=\infty$, so that for every $k\geq 1$,
$$\E_k(\w{T}_0)\geq \E_0(\w{T}_0)=\infty,$$
since the IBPRE $(\w{Z}_n)_{n\in\N}$ starting from $k\geq 1$ is stochastically larger than 
$(\w{Z}_n)_{n\in\N}$ starting from $0$. \\ \\
Moreover, under $\P_k$, $(Z_n)_{n \in\N}$ is equal to
$(\w{Z}_n)_{n\in\N}$ until time $T_0=\w{T}_0$. So $\E_k(T_0)=\infty$.
 Let $k\geq 1$ such that
$\P_0(Z_1=k)>0$, then $\E_0(T_0)\geq \P_0(Z_1=k)\E_k(T_0-1)$. This
entails that
$$\E_0(T_0)=\infty.$$
By Lemma \ref{mark} (ii),  $(Z_n)_{n \in\N}$ starting from any $k\in\N$ tends to $\infty$ in probability. \\ \\

CASE 4. Denote by
$$X_i:=\P(Z_{i}>0 \ \vert \ Z_{i-1}=1, \ f_{i-1}), \quad (i\geq1),$$
the survival probability in environment $f_{i-1}$ and introduce the following random walk
$$S_n=\sum_{i=1}^{n}\log(X_i).$$
Then
$$\P_1(Z_n>0 \ \vert \ (f_0,f_1,...,f_{n-1}))\geq \Pi_{1}^n X_i=\exp(S_n) \quad \t{a.s.},$$
so that
\Bea
\P_k(Z_n>0  \ \vert \ (f_0,f_1,...,f_{n-1}))&=& 1-\P_k(Z_n=0  \ \vert \ (f_0,f_1,...,f_{n-1})) \\
&=&1-[1-\P_1(Z_n>0  \ \vert \ (f_0,f_1,...,f_{n-1}))]^k \\
&\geq & 1-[1-\exp(S_n)]^k \quad \t{a.s.} \Eea Thus
$$\P_k(Z_n>0)\geq \E(1-[1-\exp(S_n)]^k).$$
Using  the Markov property
 we have
\Bea
\E_0(T_0+1)&\geq & \sum_{k=1}^{\infty} \P(Y_0=k)\E_k(T_0) \\
&= & \sum_{k=1}^{\infty} \P(Y_0=k) \sum_{n=1}^{\infty} \P_k(T_0\geq n) \\
&\geq & \sum_{k=1}^{\infty} \P(Y_0=k)\sum_{n=1}^{\infty} \P_k(Z_n>0) \\
&\geq & \sum_{k=1}^{\infty} \P(Y_0=k)\sum_{n=1}^{\infty}
\E(1-[1-\exp(S_n)]^k). \Eea Moreover for all  $x\in[0,1[$ and $k\geq
0$, $\exp(k\log(1-x))\leq \exp(-kx),$  and by the law of large
numbers, $S_n/n$ tends a.s. to $\E(X_1)<0$ so that there exists $n_0\geq 1$ such that for every $n\geq n_0$,
$$\P(S_n/n \geq 3\E(X_1)/2)\geq 1/2.$$
We  get then
\Bea
\E_0(T_0+1)
&\geq &  \sum_{k=1}^{\infty} \P(Y_0=k)\sum_{n=1}^{\infty} \E(1-\exp(-k\exp(S_n))) \\
&\geq &  [1-e^{-1}] \sum_{n=1}^{\infty} \sum_{k=1}^{\infty} \P(k\exp(S_n)\geq 1)\P(Y_0=k)  \\
&\geq & [1-e^{-1}] \sum_{n=n_0}^{\infty}\P(S_n/n \geq 3\E(X_1)/2) \sum_{k\geq \exp(-3n\E(X_1)/2)}^{\infty}
\P(Y_0=k)  \\
&\geq & 2^{-1}[1-e^{-1}] \sum_{n=n_0}^{\infty}
 \P(Y_0 \geq \exp(-3n\E(X_1)/2))\\
 &\geq & 2^{-1}[1-e^{-1}] \sum_{n=n_0}^{\infty}
 \P(\beta\log(Y_0) \geq n),
\Eea where $\beta:=[-3\E(X_1)/2]^{-1}>0$. Then
$\E(\log(Y_0))=\infty$  ensures that $\E_0(T_0+1)=\infty$, so
$$\E_0(T_0)=\infty.$$
Conclude that $(Z_n)_{n\in\N}$ tends to $\infty$ in $\P_k$ using Lemma \ref{mark} (ii).
\end{proof}
$\newline$

\section{Asympotics for proportions of cells with a given number of parasites}
\label{prop}

\subsection{Asymptotics  without contamination}
\label{wcont} Here there is no contamination, i.e. $Y_0=Y_1=0$ a.s.
and we determine when the organism recovers, meaning that the number
of contaminated cells becomes negligible compared to the total
number of cells. We get  the same result as Theorem 1 in
\cite{vbk} for the more general model considered here. Denote by
$N_n$ the number of contaminated cells.
\begin{Pte} \label{propsscont}$N_n/2^n$ decreases as $n$ grows. \\
If $\E(\log(f'(1)))\leq 0$, then $N_n/2^n\rightarrow 0$ a.s.  as $n\rightarrow \infty$. \\
Otherwise, $N_n/2^n\rightarrow 0$ as $n\rightarrow \infty$ iff all parasites die out, which
happens with a probability less than $1$.
\end{Pte}
\begin{ex}
Consider the case of the random binomial repartition of parasites
mentioned in Introduction. Let $Z \in\N$ be a r.v and $(P_{\i})_{\i
\in\TT}$ be an i.i.d. sequence distributed as a  r.v. $P\in[0,1]$, such
that $P\stackrel{d}{=}1-P$. In every generation, each parasite gives
birth independently to a random number of parasites distributed as
$Z$. When the cell $\i$ divides, conditionally on $P_{\i}=p$,  each
parasite of the cell $\i$  goes independently in the first daughter
cell with probability $p$ (or it goes in the second daughter cell,
which happens with probability $1-p$). Then,
$$\P(f'(1) \in \d x)=\P(\E(Z)P \in \d x).$$
Thus, the organism recovers a.s. (i.e. $N_n/2^n$ tends a.s. to $0$)
iff
$$\log(\E(Z))\leq \E(\log(1/P)).$$
This is the same criteria in the case when the offspring of each
parasite goes a.s. is the same  daughter cell (there, $p$
is the probability that this offspring goes in the first daughter cell.)  \\
\end{ex}
$\newline$
\begin{proof}
Note that $N_n/2^n$ decreases to $L$ as $n\rightarrow \infty$, since
one infected cell has at most two daughter cells which are infected.
Moreover,   for every $n\in \N$, \Bea
\E\left(\frac{N_n}{2^n}\right)&=&\frac{\E(\sum_{i\in\GG_n} \ind_{Z_{\b{i}}>0})}{2^n} \\
&=& \sum_{\b{i}\in\GG_n} \frac{1}{2^n} \E(\ind_{Z_{\b{i}}>0}) \\
&=& \sum_{\b{i}\in\GG_n} \P((a_0,...,a_{n-1})=\i) \P(Z_{\b{i}}>0) \\
&=&\P(Z_n>0).
\Eea

If $\E(\log(f'(1)))\leq 0$ (subcritical or critical case), then
$\P(Z_n>0)$ tends to $0$ as $n\rightarrow \infty$ (see Section \ref{BPRE}).  Thus,
$\E(L)=0$ and $N_n/2^n$ tends to $0$ a.s. as $n\rightarrow\infty$. \\ 

If $\E(\log(f'(1)))> 0$ (supercritical case), then $\P(Z_n>0)$ tends to a positive value, which is equal to $\P(L>0)0$. We complete the proof with the following lemma.
\end{proof}




Let us prove the following zero one law, where $P_n$ is the total number of parasites in generation $n$.
\begin{Lem}
If $\E(\log(f'(1)))> 0$, then  $$\{\lim_{n\rightarrow \infty} N_n/2^n >0\}=\{\forall n \in \N: P_n>0\} \quad \t{a.s.}$$  
\end{Lem}
\begin{proof}
First, we prove that conditionally on non-extinction of parasites, for every $K \in \N$, there
exists a.s. a generation $n$ such that $N_n\geq K$. Letting $K \in \N$, we fix $p$ as the first integer such that $2^p\geq K$. Then 
$q:=\P_1(N_p\geq K)>0$ since $\P(N_1=2)>0$. \\
Either the number of infected cells in generation $p$ is more than $K$, 
which happens with probability $q$, or we can choose in generation
$p$ an infected cell $\b{i}(1)$, since parasites have not died out. Then, with probability larger than $q$ ,
the number of infected cells in generation $p$ of  the subtree
rooted in this cell $\b{i}(1)$ contains more than $K$ parasites. Note that this
probability is exactly equal to $q$ iff the infected cell $\b{i}(1)$ contains one single
parasite. Recursively, we find a.s. a generation $n$ such than $N_n\geq K$. \\

Then, recalling that we still work conditionally on  non-extinction of parasites, the stopping time 
$T:=\inf\{n \in \N: N_n\geq K\}<\infty$ a.s. We now also condition 
by $T=n$ and $N_T=k$. We can then choose one parasite in every infected cell  in generation $n$, which we label by  $1\leq i \leq k$ and we denote   by $N_p^{(i)}$ the
number of cells in generation $n+p$ infected by parasites  whose ancestor in
generation $n$ is the parasite $i$. By branching property,  the integers $(N_p(i):  \ 1\leq i\leq k)$ are i.i.d.
and $N_p^{(i)}/2^p\rightarrow L^{(i)}$ as $p\rightarrow \infty$,
where $(L^{(i)}: \ 1\leq i\leq k)$ are independent and
$\P(L^{(i)}>0)= \P(L>0)>0$ for every $1\leq i \leq k$. Using that
$$N_{n+p}\geq \sum_{i=1}^k N_p^{(i)} \qquad  \t{a.s.},$$
and as $k\geq K$, we get
$$\lim_{p\rightarrow \infty} N_{n+p}/2^p\geq \max(L^{(i)}: \ 1\leq i\leq K) \quad \t{a.s.} $$
As $\sup(L^{(i)}: \ i\in\N)=\infty $ a.s., letting $K\rightarrow \infty$ 
ensures that a.s. $N_p/2^p$ does not tend to $0$.
\end{proof}
$\newline$
\subsection{Asymptotics with contamination in the case $\E(\log(f'(1)))<0$ and $\max(\E(\log^+(Y_i)): i=0,1)<\infty$.}
\label{cont}

Define $F_k(n)$ the  proportion
  of cells with $k$ parasites in generation $n$:
$$F_k(n):= \frac{ \# \{ \mathbf{i}\in \GG_n:
Z_{\mathbf{i}}=k\}}{2^n} \qquad (k\in\N).$$
We introduce the Banach  space $l^1(\N)$ and  the subset of
frequencies $\SS^1(\N)$ which we endow with  the norm $\parallel
.\parallel _1$ defined by:
$$l^1(\N):=\{(x_i)_{i\in\N}: \sum_{i=0}^{\infty} \vert x_i\vert <\infty\},
\qquad \parallel (x_i)_{i\in\N} \parallel _1=\sum_{i=0}^{\infty} \vert x_i\vert,$$
$$\SS^1(\N):=\{(f_i)_{i\in\N}: \forall \ i \in \N, \  f_i \in \RRR^+, \ \sum_{i=0}^{\infty} f_i=1\}. $$
The main argument here is the law of large number proved by Guyon
\cite{guyon} for asymmetric Markov chains indexed by a tree. 
$\newline$
 
\begin{Thm}
If $\E(\log(f'(1)))<0$ and $\max(\E(\log^+(Y_i)): i=0,1)<\infty$,
then $(F_k(n))_{k\in\N}$ converges in probability in $\SS^1(\N)$  to
a deterministic sequence  $(f_k)_{k\in\N} $ as $n\rightarrow
\infty$, such that $f_0>0$ and $\sum_{k=0}^{\infty} f_k=1$.
Moreover, for every $k\in\N$, $f_k=\P(Z_{\infty}=k)$.
\end{Thm}
\begin{proof}Recall that $(Z_{\i})_{\i\in\TT}$ is a Markov chain indexed by a tree and we are in the framework of bifurcating Markov chain studied in \cite{guyon}. Thanks
to the ergodicity of the number of parasites in a random cell line
proved in the previous section (Theorem \ref{cvrandline} (i)), we can directly apply Theorem 8 in
\cite{guyon} to get the convergence of proportions
of cells with a given number of parasites.
\end{proof}
$\newline$
But it seems that we can't apply Theorem 14 or
Corollary 15 in \cite{guyon} to get a.s. convergence of proportions, because of the term $k^{\epsilon}$ in estimation of Theorem \ref{cvrandline}. For examples, we refer to the previous proposition. \\ \\


Using again \cite{guyon}, we can prove also a law of large
 numbers and a  central limit theorem for the proportions of
 cells with given number of parasites before generation $n$. Define,
 for every $n\in\N$,
$$P_k(n):=\frac{ \# \{ \mathbf{i}\in \cup_{0\leq i\leq n} \GG_i:
Z_{\mathbf{i}}=k\}}{2^{n+1}} \qquad (k\in\N).$$
\begin{Thm}
If $\E(\log(f'(1)))<0$ and $\max(\E(\log^+(Y_i)): i=0,1)<\infty$, then
$(P_k(n))_{k\in\N}$
converges in probability in $\SS^1(\N)$  to the deterministic sequence  $(f_k)_{k\in\N} $ as $n\rightarrow \infty$. \\
Moreover for every $k\in\N$, $\sqrt{n}(P_k(n)-f_k)$ converges in
distribution to a centered normal  law as $n\rightarrow \infty$, with a non explicit variance.
\end{Thm}
\begin{proof}
Use again Theorem \ref{cvrandline} (i) and  Theorem 8 in \cite{guyon}
to prove the law of large numbers. For the central limit theorem, use  Theorem 19 in
\cite{guyon} by letting $F$ be the set of continuous functions
taking values in $[0,1]$.
\end{proof}

$\newline$
\subsection{Asymptotics with contamination in the case $\E(\log(f'(1)))\geq 0$ or
$\max(\E(\log^+(Y_i)): i=0,1)=\infty$.} 
\label{supercont}
In this case, cells become
infinitely infected as the generation tends to infinity.
\begin{Thm} If $\E(\log(f'(1)))\geq 0$ or
$\max(\E(\log^+(Y_i)): i=0,1)=\infty$,
for every $k\in\N$, then $F_k(n)$ tends to zero as $n\rightarrow \infty$.
That is, for very $K \in \N$,
$$ \lim_{n\rightarrow \infty} \#\{ \i \in \GG_n: Z_{\i}\geq K\}/2^n \stackrel{\P}{=}1.$$
\end{Thm}
\begin{proof} By Fubini's theorem, we have
\Bea
\E\big[\#\{ \i \in \GG_n: Z_{\i}\geq K\}/2^n\big]&=&\sum_{i\in\GG_n} \P(Z_{\i}\geq K)/2^n \\
&=& \sum_{i\in\GG_n} \P((a_0,...a_{n-1})=\i)\P(Z_{\i}\geq K) \\
&=&\P(Z_n\geq K).
\Eea
By Theorem \ref{cvrandline}, $\P(Z_n\geq K)$ tends to $1$, then
$1-\#\{ \i \in \GG_n: Z_{\i}\geq K\}/2^n $ converges to $0$ in $L^1$, which  gives the result.
\end{proof}
$\newline$
\section{Asymptotics for the number of parasites}
\label{nbparasit}
We assume here that parasites multiply following a Galton Watson process with deterministic  mean $m$, independently of the cell they belong to. That  is, $s\mapsto  \f(s,s)$ is deterministic and every parasite multiply independently with the reproduction law whose p.g.f. is equal to $g : s\mapsto \f(s,s)$. Moreover we assume that
contamination of a cell does not depend on the number of enclosed parasites. That is
$$Y\stackrel{d}{=}Y_0\stackrel{d}{=}Y_1.$$

Set $P_n$ the number of parasites in generation $n$.  Without contamination, in the supercritical case $m>1$,  it is well know that either $P_n$ becomes extinct or $P_n/m^n$ converges to a positive finite random variable. In the presence of contamination, we have the following result.
\begin{Pte}
If $\E(Y)<\infty$ and $\P(Y_0>0)>0$, then $\log(P_n)/n$ converges in $\P$
 to $\log(\max(2,m))$.
\end{Pte}
\begin{proof}
First, we prove the lower bound. This is a  consequence of the fact that $P_n$ is larger than\\
(i) the total number of parasites $P_n^1$  which contaminate  cells of generation $n$, \\ 
(ii) the number of parasites $P_n(p)$ in generation $n$ with the same given parasite ancestor in generation $p$. 

Indeed, first $P_n^1$ is the sum of $2^n$ i.i.d. random variables with mean $\E(Y)$, so law of large numbers ensures that 
$$P_n^1/2^n \stackrel{n\rightarrow \infty}{\longrightarrow } \E(Y)>0 \quad \t{in} \ \P$$
Then, since $P_n\geq P_n^1$ a.s. for every $n\in\N$, 
\be
\label{Pninf}
P_n  \stackrel{n\rightarrow \infty}{\longrightarrow } \infty \quad \t{in} \ \P.
\ee
Moreover for every $p<n$,
\be 
\label{renormalisation}
P_n(p)/m^{n-p}\stackrel{n\rightarrow \infty}{\longrightarrow } W, \qquad \t{a.s.},
\ee
with $\P(W>0)>0$. Let now $P_n^2$ be the sum of the number of descendants in generation $n$ of each parasite of generation $p$. We get then the sum of $P_p$ i.i.d. quantities distributed as $P_n(p)$. Then
$(\ref{Pninf})$ and (\ref{renormalisation}) ensure
that we can choose $p$ such that 
$$P_n^2/m^{n-p}  \stackrel{n\rightarrow \infty}{\longrightarrow } W' \qquad \t{a.s.}$$
with $\P(W'>0)\geq 1-\epsilon$.

 Using that $N_n$ is larger than $P_n^1$ and $P_n^2$ ensures that for every $\epsilon>0$,
$$\limsup_{n\rightarrow \infty} \P(\log(P_n)/n\leq \log(\max(2,m)))< \epsilon.$$
Letting $\epsilon\rightarrow 0$ gives the lower bound. \\

Second, we prove the upper bound. Note that the total number of parasites in generation $n$ can be written as
$$P_n=\sum_{i=1}^n \sum_{j=1}^{2^i} \sum_{k=1}^{Y^{i,j}} Z^{i,j}_k,$$
where $Y^{i,j}$ is the number of parasites which contaminate the $j$th cell of generation $i$, and labeling by
$1\leq k\leq Y^{i,j}$ these parasites, $Z^{i,j}_k$
is the number of descendants in generation $n$ of the $k$th parasites. \\
Moreover $(Y^{i,j}:  i \in \N, j \in \N)$ are identically distributed and independent  of $(Z^{i,j}_p(k),i \in \N, j \in \N, \ k \in \N)$,  $(Z^{i,j}_k,i \in \N, j \in \N, \ k \in \N)$  are independent
and $Z^{i,j}_p(k)$  is the population of a Galton  Watson process in   generation $n-i$ with offspring probability generation function equal to $g$. Thus
\Bea
\E(P_n)&=& \sum_{i=1}^n \sum_{j=1}^{2^i} \E(\sum_{k=1}^{Y^{i,j}} Z^{i,j}_k) \\
&=& \sum_{i=1}^n \sum_{j=1}^{2^i} \E(Y^{i,j}) \E( Z^{i,j}_k) \\
&=&  \E(Y) \sum_{i=1}^n \sum_{j=1}^{2^i} m^{n-i} \\
&=& 2\E(Y) \frac{m^n-2^n}{m-2} \quad \t{if} \ m\ne 2.
\Eea
If $m=2$, then $\E(P_n)=\E(Y)nm^n$.
This gives the upper bound by Markov inequality and completes the proof.
\end{proof}


\begin{thebibliography}{19}
\bibitem{asmu} Asmussen S.,  Hering H. (1983). \emph{Branching processes.} Progress in Probability and Statistics, 3. Birkhäuser Boston, Inc., Boston, MA.
\bibitem{at} Athreya K. B.,  Karlin S. (1971).  On branching processes with random environments, I : extinction probability \emph{Ann. Math. Stat.} \textbf{42}. 1499-1520.
\bibitem{atr} Athreya K. B., Karlin S. (1971).  On branching processes with random environments, II : limit theorems. \emph{Ann. Math. Stat.} \textbf{42}. 1843-1858.
\bibitem{AN}   Athreya K. B.,  Ney P. E. (2004). \emph{Branching processes}. Dover Publications, Inc., Mineola, NY.
\bibitem{atarbre} Athreya K. B.,  Kang H. J.  (1998). Some limit theorems for positive recurrent branching Markov chains I. \emph{Adv. Appl. Prob.} \textbf{30} (3). 693-710.
\bibitem{atarbres} Athreya  K.,  Kang H.J. (1998). Some limit theorems for positive recurrent branching Markov chains II. \emph{Adv. Appl. Prob.} \textbf{30} (3). 711-722.
\bibitem{agkv}  Afanasyev V. I.,  Geiger J.,  Kersting G.,  Vatutin  V. (2005). 
Criticality for branching processes in random environment.  
\emph{Ann. Probab.} \textbf{33} , no. 2, 645-673. 
\bibitem{vbk} Bansaye V. (2008). Proliferating parasites in dividing cells : Kimmel's branching model revisited.  \emph{Ann. Appl. Prob.} \textbf{18} (3), 967-996.
\bibitem{benj} Benjamini I., Peres Y. (1994). Markov chains indexed by  trees. \emph{Ann. Probab.} \textbf{22} (1). 219-243.
\bibitem{bpree} Geiger J., Kersting G.,  Vatutin V. A. (2003). Limit Theorems for subcritical branching processes in random environment. \emph{Ann. I. H. Poincaré} \textbf{39}, no. 4.  593-620.
\bibitem{Feller} Feller W. (1966). \emph{An introduction to probability theory and its applications}. Vol. II. John Wiley  Sons, Inc., New York-London-Sydney. 
\bibitem{guyon}  Guyon J. (2007). Limit theorems for bifurcating Markov chains. Application to the detection of cellular aging. \emph{Ann. Appl. Probab.} \textbf{17}, no. 5-6, 1538-1569.
\bibitem{IBPRE} Key E. S. (1987). Limiting Distributions and Regeneration Times for Multitype Branching Processes with Immigration in a Random Environment. \emph{Ann.   Prob.}, Vol. 15, No. 1, 344-353.
\bibitem{kendall}  Kendall   D. (1959). Unitary dilations of Markov transition operators, and the corresponding integral representations for transition-probability matrices. \emph{Probability and statistics} : The Harald Cramér volume (edited by Ulf Grenander.) 139-161.
\bibitem{kim}  Kimmel M. (1997). Quasistationarity in a branching model of division-within-division. \emph{Classical and modern branching processes (Minneapolis, MN, 1994)}. 157-164, IMA Vol. Math. Appl., 84, Springer, New York.
\bibitem{koz}  Kozlov M. V. (1976). The asymptotic behavior of the probability of non-extinction of critical branching processes in a random environment. \emph{Theor. Probability Appl.} \textbf{21}, no. 4, 813--825.
\bibitem{Roit} Roithershtein A. (2007).  A note on multitype branching processes with immigration in a random environment. \emph{Ann. Probab.}
 Vol. 35, No. 4, 15731592.
\bibitem{RuL}  Lyons R.,  Peres Y (2005). \emph{Probability on Trees and Networks}.  Available via http://mypage.iu.edu/~rdlyons/prbtree/prbtree.html.
\bibitem{bpre} Smith W. L., Wilkinson W. (1969). On branching processes in random environments. \emph{Ann. Math. Stat.} \textbf{40}, No 3, p814-827.
\bibitem{tad}  Stewart E. J.,  Madden R., Paul G., Taddei F. (2005). Aging and Death in a Organism that reproduces by Morphologically Symmetric Division. \emph{PLoS Biol}, 3 (2) : e45.
\end{thebibliography}
\end{document}